\documentclass[twoside]{article}

\usepackage{amsmath}
\usepackage{hyperref}
\usepackage{mathrsfs}
\usepackage{fancyhdr}
\usepackage{array}
\usepackage{float}
\usepackage{amscd}
\usepackage{amstext}
\usepackage{amsthm}
\usepackage{amsfonts}
\usepackage{amssymb}
\usepackage[pdftex]{graphicx}
\def\pf{\noindent{\it Proof. }}

\newcommand{\B}{{\mathbb B}}

\newcommand{\EE}{\mathscr E}
\newcommand{\N}{{\mathbb N}}
\newcommand{\NN}{\mathscr N}

\newcommand{\R}{{\mathbb R}}
\newcommand{\Sp}{{\mathbb S}}
\newcommand{\T}{{\mathbb T}}
\newcommand{\Z}{{\mathbb Z}}

\newcommand{\RR}{\mathbb{R}}

\newcommand{\BB}{\mathbb{B}}
\newcommand{\G}{\mathbb{G}}
\newcommand{\g}{\mathfrak{g}}
\providecommand{\abs}[1]{\lvert#1\rvert}
\providecommand{\norm}[1]{\lVert#1\rVert}

\newtheorem{theorem}{Theorem}[section]

\newtheorem{lemma}{Lemma}[section]

\newtheorem{proposition}{Proposition}[section]

\begin{document}

\begin{titlepage}

\title{On Weyl's asymptotics and remainder term for the orthogonal and unitary groups}

\author{Charles Morris, Ali Taheri$^\dag$}

\date{}

\end{titlepage}

\maketitle

\begin{abstract}
We examine the asymptotics of the spectral counting function of a compact Riemannian 
manifold by V.G.~Avakumovic \cite{Avakumovic} and L.~H\"ormander \cite{Hormander-eigen} and 
show that for the scale of orthogonal and unitary groups ${\bf SO}(N)$, 
${\bf SU}(N)$, ${\bf U}(N)$ and ${\bf Spin}(N)$ it is not sharp. While for negative sectional curvature 
improvements are possible and known, {\it cf.} e.g., J.J.~Duistermaat $\&$ 
V.~Guillemin \cite{Duist-Guill}, here, we give sharp and contrasting examples in the positive Ricci curvature 
case [non-negative for ${\bf U}(N)$]. Furthermore here the improvements are sharp and 
quantitative relating to the dimension and {\it rank} of the group. We discuss the implications 
of these results on the closely related problem of closed geodesics and the length spectrum.
\end{abstract}
\vspace{10pt}



\section{Introduction}
\setcounter {equation}{0}

Let $(M^d, g)$ be a $d$ dimensional compact Riemannian manifold without boundary and let $-\Delta_g$ denote 
the Laplace-Beltrami operator on $M^d$. The spectral counting function of $-\Delta_g$ on $M^d$ is the function 
defined on %the half-axis 
$(0, \infty)$ by 
\footnote{In the sequel when the choice of $M^d$ is clear from the context, or when 
there is no danger of confusion, we suppress the dependence on $M^d$ and simply write $\mathscr{N}=\mathscr{N}(\lambda)$.} 
\begin{equation} \label{spectral-counting-function-equation}
\mathscr{N}(\lambda; M^d) = \sharp \{j \ge 0 : \lambda_j \le \lambda \}, \qquad \lambda > 0.
\end{equation}
Here $0=\lambda_0 < \lambda_1 \le \lambda_2 \le  ... $ denote the eigen-values of $-\Delta_g$ in ascending order where by 
basic spectral theory each eigen-value has a finite multiplicity while $\lambda_j \nearrow \infty$ as $j \nearrow \infty$. 
The description of the asymptotics of the spectral counting function has been the subject of numerous investigations and is a 
result of a sequence of improvements and refinements starting originally from the seminal works of H.~Weyl in 1911 
describing the leading term and then gradually sharpening the form and order of the remainder term through the works of various 
authors most notably B.M.~Levitan \cite{Levitan}, V.G.~Avakumovic \cite{Avakumovic} and L.~H\"ormander \cite{Hormander-eigen}. 
(See also the monographs V.~Ivrii \cite{Ivrii-2} and M.A.~Shubin \cite{Shubin}.) 
For compact boundaryless manifolds the celebrated Avakumovic-H\"ormander-Weyl asymptotics has the from
\begin{equation} \label{Weyl-law-remainder-general-equation} 
\mathscr{N}(\lambda; M^d) = \frac{{\rm Vol}_g(M^d) \omega_d}{(2\pi)^d} \lambda^\frac{d}{2} + O \left( \lambda^\frac{d-1}{2} \right), 
\qquad \lambda \nearrow \infty,
\end{equation}
where ${\rm Vol}_g(M^d)$ is the volume of $M^d$ with respect to $dv_g$ and $\omega_d$ is the volume of the unit 
$d$-ball in the Euclidean space $\R^d$, that is, $\omega_d=|\B^d_1|$. %(We point out that there is also a counterpart of Weyl's law for 
%manifolds {\it with} boundary that we do not enter here. For more discussion on this {\it see} V.~Ivrii \cite{Ivrii-1}-\cite{Ivrii-2} and the 
%references therein particularly to the works of T.~Carleman and R.~Melrose among others.)
The formulation of (\ref{Weyl-law-remainder-general-equation}) is originally due to V.G.~Avakumovic \cite{Avakumovic} and later on 
L.~H\"ormander \cite{Hormander-eigen}, who by invoking the theory of Fourier integral operators and the wave equation gave a proof 
that simultaneously extends to operators of arbitrary order ({\it see} \cite{Duist-Horm, Horm}, \cite{Hormander} Vols.~{\bf 3}-{\bf 4} or 
\cite{Shubin}). That the remainder term is sharp can be seen by examining, e.g., Euclidean spheres or projective spaces (see below 
for more on this) whilst in contrast, for flat tori, the problem directly relates to counting integer lattice points and is far from sharp. 
Indeed recall that for ${\mathbb T}^d = \R^d /\Z^d$ the eigen-functions $(\varphi_\alpha : \alpha \in \Z^d)$ are characterised by 
\begin{align}
-\Delta_g \varphi_\alpha &= 4 \pi^2 |\alpha|^2 \varphi_\alpha, \qquad \varphi_\alpha = e^{2\pi i \alpha \cdot x},  \\
\alpha &= (\alpha_1, ..., \alpha_d), \qquad |\alpha|^2 = \alpha_1^2 + ... + \alpha^2_d,  \nonumber 
\end{align}
and so as a result for $\lambda>0$ and with $\chi$ the characteristic function of the closed ball centred at the origin 
and radius $r=\sqrt \lambda/2\pi$ we have 
\footnote{For further improvements of this classical result of E.~Hlawka \cite{Hlawka} and more on the lattice point problem see 
F.~Chamizo $\&$ H.~Iwaniec \cite{ChamIwan}, D.R.~Heath-Brown \cite{HB}, M.N.~Huxley \cite{Huxley}, A.~Walfisz \cite{Wal} 
as well as F.~Fricker \cite{Fricker}. See also Section \ref{SecTwo} and Table \ref{table delta} below.}
\begin{align} \label{flat-tori-scf-equation}
\mathscr{N}(\lambda; {\mathbb T}^d) &= \sharp \bigg\{x \in \Z^d : 4 \pi^2 |x|^2 \le \lambda \bigg\}
= \sum_{x \in \Z^d} \chi (x) \nonumber \\
&= \frac{{\rm Vol}_g({\mathbb T}^d) \omega_d}{(2\pi)^d} \lambda^\frac{d}{2} + 
O \left( \lambda^{\frac{d-1}{2} - \frac{1}{2} \frac{d-1}{d+1}} \right). 
\end{align}

A natural question thus is when is the remainder term in (\ref{Weyl-law-remainder-general-equation}) sharp 
and if the sharpness of this term carries any geometric information. To elaborate on this further consider the half-wave equation 
\begin{equation} 
\left \{ \begin {array}{ll}
\partial_t u + i A u =0 & \textrm{$t \in \R$,} \\
\qquad \quad \,\,\,\,\, u=f & \textrm{$t=0$,}   
\end{array} \right .  
\end{equation} 
associated with $A=\sqrt{-\Delta_g}$ in $L^2(M)$. Then $u(t) = e^{-it A} f$ 
and so upon taking traces and with $\mathscr{N}_A(\mu) = \sharp \{j : \mu_j=\sqrt{\lambda_j} \le \mu\}=\mathscr{N}(\mu^2)$ we have
\begin{equation}
{\rm Tr} \left( e^{- it \sqrt{-\Delta_g}} \right) =  \sum e^{-it \sqrt{\lambda_j}} = \widehat{d\mathscr{N}_A} (t/2\pi).
\end{equation}

The half-wave propagator $e^{-it \sqrt{-\Delta_g}}$ can be expressed as a Fourier integral operator and in view of the above identity its 
trace is the Fourier transform of the measure $d \mathscr{N}_A$. Thus by basic considerations the asymptotics of this trace near $t=0$ translates 
via a Fourier inversion (a trivial but more revealing Tauberian theorem in this context) to the asymptotics of the spectral counting function 
as $\lambda \nearrow \infty$.

Now consideration of this Fourier integral operator, its canonical relation and the Lagrangian flow associated to its principal symbol leads 
to the geodesic flow on the cotangent bundle $T^\star M$ of $M$ (note that it is precisely here that one needs to deal with the first order 
operator $\sqrt{-\Delta_g}$). The periodic orbits of this Lagrangian flow are the periodic geodesics on $M$ and the resulting length spectrum 
contains the singular support of the distributional trace 
\begin{equation}
\rho(t)={\rm Tr}(e^{-i t \sqrt{-\Delta_g}})=\sum_{j} e^{-it \sqrt{\lambda_j}}, \qquad -\infty<t<\infty.  
\end{equation}

Indeed the analysis by L.~H\"ormander of the big singularity of $\rho$ at $t=0$ leads to the trace formula of 
J.J.~Duistermaat $\&$ V.~Guillemin \cite{Duist-Guill}
\begin{equation}
\sum_j h(\mu - \mu_j) \cong \frac{1}{(2\pi)^d} \sum_{k=0}^{d-1} c_k \mu^{d-1-k}, \qquad c_k = \int_{M^d} \omega_k, \qquad \mu \nearrow \infty,
\end{equation}
where $\mu_j=\sqrt{\lambda_j}$ and $h \in {\mathcal S}(\R)$ with ${\rm supp}\, \hat h \subset [-\varepsilon, \varepsilon]$ while $h \equiv 1$ 
in a suitably small neighbourhood of zero and
%\begin{equation}
%c_k = \int_M \omega_k,
%\end{equation}
%with $\omega_k$ being 
$\omega_k$ smooth real-valued densities on $M^d$ associated to the metric $g$. (In particular we have $c_0={\rm Vol}(\B^\star M)$ 
the volume of the unit ball in the co-tangent bundle.)

The sharpness of the remainder term in (\ref{Weyl-law-remainder-general-equation}) now 
connects directly with the structure of the spectrum $(\lambda_j: j \ge 0)$ and the nature of the geodesic flow. For example in case of Euclidean 
spheres, real or complex projective spaces or more generally compact rank one symmetric spaces the spectrum clusters, the geodesic flow is 
periodic and the remainder term in (\ref{Weyl-law-remainder-general-equation}) is sharp. However, and in contrast, for spaces with non-positive sectional 
curvature or more generally spaces with measure-theoretically {\it few} periodic geodesics the remainder term in 
(\ref{Weyl-law-remainder-general-equation}) is {\it not} sharp and can be improved 
to $O(\lambda^{(d-1)/2}/\log \lambda)$ and $o(\lambda^{(d-1)/2})$ respectively. 
({\it See} P.~Berard \cite{Berard}, J.J.~Duistermaat $\&$ V.~Guillemin \cite{Duist-Guill} as well as \cite{AwonusikaTah}, 
V.~Ivrii \cite{Ivrii-2}, C.~Sogge \cite{Sogge-3} for more.)

In this paper motivated, by the significance of the period geodesics on the Lie group ${\mathbb G}={\bf SO}(N)$ in providing twist solutions 
to certain geometric problems in the calculus of variations ({\it see} \cite{Shah-Tah, Taheri, Taheri-NoDEA}) we take a closer look at the 
geodesic length spectrum and the spectral counting function and examining the sharpness of the remainder term in Weyl's law 
(\ref{Weyl-law-remainder-general-equation}) for this case as well as the cases of the unitary and spinor groups. 
By analogy with the case of symmetric spaces and in contrast to the negative curvature case above, one expects, in virtue of positivity of 
${\rm Ric}(\G)$, that again the remainder term is sharp, however, we show this {\it not} to be the case. As a prototype example for the special 
orthogonal group ${\bf SO}(N)$ apart from $N=3$ where the geodesic flow is periodic -- note that ${\bf SO}(3) \cong {\mathbb P}(\R^3)$ is a 
rank one symmetric space -- the remainder term in (\ref{Weyl-law-remainder-general-equation}) is not sharp and can be quantitatively 
improved.

Indeed the spectral counting function $\mathscr{N}=\mathscr{N}(\lambda; \G)$ of the orthogonal and unitary 
group $\G$ equipped with a bi-invariant metric $g$ has the asymptotics:
%compact Lie group $\G={\bf SO}(N), {\bf SU}(N), {\bf U}(N), {\bf Spin}(N)$ equipped with a bi-invariant metric $g$ has the asymptotics:
 \begin{equation}
\mathscr{N}(\lambda; \G) = \frac{\omega_d {\rm Vol}_g(\G)}{(2\pi)^d} \lambda^\frac{d}{2} 
+ O \left(\lambda^{\frac{1}{2} \left[ d -1 - \varepsilon \right]  } \right), 
\qquad \lambda \nearrow \infty, 
\end{equation}
where $d={\rm dim}(\G)$, $n={\rm rank}(\G)$ ({\it see} Tables \ref{table vol 1} $\&$ \ref{table vol 2} below) and $\varepsilon=\varepsilon(n) 
\ge 0$. In fact we show that $\varepsilon=1$ when $n \ge 5$ and more generally $\varepsilon= (n-1)/(n+1)$ when $n \ge 1$. In particular 
when ${\rm rank}(\G) \ge 2$ the Avakumovic-H\"ormander-Weyl remainder term in \eqref{Weyl-law-remainder-general-equation} is not 
sharp and as direct calculation (as in Section \ref{SecTwo} below) reveals this is precisely when the geodesic flow of $\G$ {\it fails} to be 
periodic. More interestingly when ${\rm rank}(\G) \ge 5$ the exponent of $\lambda$ in the remainder term 
\eqref{Weyl-law-remainder-general-equation} can be improved to $(d-2)/2$ which is sharp. 
%which is the sharpest one can get ({\it cf.}, e.g., C.~Sogge \cite{Sogge-1}). 

Let us end this introduction by giving a brief plan of the paper. In Section \ref{SecTwo} we go over the main results and tools from the representation 
theory of compact Lie groups, in particular the computation of the Casimir spectrum, root systems and the analytic weights of irreducible 
representations followed by calculations relating to periodic geodesics and the length spectrum. In the interest of brevity and for the sake of 
definiteness the discussion here is confined to the prototype case of the special orthogonal group ${\bf SO}(N)$. In Sections \ref{SecThree} and 
\ref{SecFour} we give detailed analysis of the spectral counting function and its asymptotics for the scale of special orthogonal, unitary and spinor 
groups. Finally in Section \ref{SecFive} we present and prove the sharper form of the asymptotics of $\mathscr{N}=\mathscr{N}(\lambda; \G)$ as 
highlighted in the discussion above for $n \ge 5$.

\section{Weyl chambers and spectral multiplicities for the Lie group ${\mathbb G}={\bf SO}(N)$ with $N \ge 2$}\label{SecTwo}
\setcounter {equation}{0}

In this section we gather together some of the technical apparatus for computing and describing the spectrum and spectral 
multiplicities of the Laplace-Beltrami operator required later for the 
development of the paper. \footnote{The reader is referred to the monographs \cite{BD, Helgason} 
and \cite{Knapp} for further details and the jargon on Lie groups and their representations.} 
The Cartan-Killing form on the special orthogonal group ${\bf SO}(N)$ corresponds to 
the bilinear form $B(X,Y)=(N-2){\rm tr}(X Y)$ with $X, Y \in {\mathfrak s}{\mathfrak o}(N)$. In virtue of the semisimplicity 
of ${\bf SO}(N)$ the latter leads to an inner product on the Lie algebra ${\bf {\mathfrak s}{\mathfrak o}}(N)$, here, taking 
the explicit form 
\begin{equation} \label{Cartan-Killing-metric-equation}
(X,Y) = -\frac{B(X, Y)}{2(N-2)} = \frac{1}{2} {\rm tr}(XY^t) = \frac{1}{2} {\rm tr}(X^tY). 
\end{equation}
This inner product in turn results in a bi-variant metric on ${\bf SO}(N)$ that from now on is the choice of Riemannian metric. 
Specifically using left translations this gives the Adjoint invariant metric on $\G={\bf SO}(N)$ defined via, 
\begin{align}
\rho(X_g,Y_g) &= \frac{1}{2} {\rm tr}((g^{-1}X_g)^tg^{-1}Y_g) =\frac{1}{2} {\rm tr}(X^tY).
\end{align}
As the metric is bi-invariant the Riemannian exponential and the Lie exponential coincide (see, e.g., S.~Helgason \cite{Helgason}) 
and so the geodesic $\gamma=\gamma(t)$ starting at $g\in \G$ in the direction $X_g \in T_g \G$ is given by $\gamma(t) = g \exp(tX)$ 
where $X\in \g$ is such that $X_g = g X$. Now we fix the maximal torus $\T$ on $\G$ by setting, for even $N$,
\begin{equation}
\T = \bigg\{ g \in {\bf SO}(2n) : \mbox{$g={\rm diag}({\mathcal R}_1, {\mathcal R}_2, ..., {\mathcal R}_n)$} \bigg\},
\end{equation}
with the $2 \times 2$ rotation blocks ${\mathcal R}_j$ ($1 \le j \le n)$ in ${\bf SO}(2)$ given by ${\mathcal R}_j = {\rm exp}\,{\mathcal J}_j$:
\begin{equation*}
{\mathcal R}_j = {\mathcal R}(a_j) = 
\begin{bmatrix} 
\cos a_j & -\sin a_j \\
\sin a_j &  \quad \cos a_j
\end{bmatrix}, 
\qquad 
{\mathcal J}_j = {\mathcal J} (a_j) = \begin{bmatrix}
0 & -a_j \\
a_j & \quad 0
\end{bmatrix}, 
\end{equation*}
($a_j \in \R$) and the usual adjustment for odd $N$, namely, $n$, $2 \times 2$ block as above and a last $1 \times 1$ block 
consisting of entry $1$, specifically, 
\begin{equation}
\T = \bigg\{ g \in {\bf SO}(2n+1) : \mbox{$g={\rm diag}({\mathcal R}_1, {\mathcal R}_2, ..., {\mathcal R}_n, 1)$} \bigg\}.
\end{equation}
The subalgebra $\mathfrak{t} \subset \g$ corresponding to the maximal torus $\T \subset \G$ for even $N$ is given by,
\begin{equation} \label{subalgebra-t}
\mathfrak{t} = \bigg\{ \xi \in \g : \mbox{$\xi = {\rm diag}({\mathcal J}_1, {\mathcal J}_2, ..., {\mathcal J}_n)$} \bigg\}, \qquad 
N=2n, 
\end{equation}
and for odd $N$ by 
\begin{equation}
\mathfrak{t} = \bigg\{ \xi \in \g : \mbox{$\xi = {\rm diag}({\mathcal J}_1, {\mathcal J}_2, ..., {\mathcal J}_n, 0)$} \bigg\}, \qquad 
N=2n+1.
\end{equation}
For the geodesic $\gamma(t) = \exp(tX)$ there exists $g \in {\bf SO}(N)$ so that $X =g \xi g^{-1}$ for suitable 
$\xi \in \mathfrak{t}$ and $\gamma(t) = g \exp(t \xi) g^{-1}$. It is therefore plain that any periodic geodesic at identity 
is {\it conjugate} to one sitting entirely on the maximal torus $\T$. Thus in considering the periodic geodesics 
of $\G$ we can merely focus on those confined to $\T$. Now let $\Lambda \subset {\mathfrak t}$ denote the lattice 
\begin{equation} \label{lattice-L-def-equation}
\Lambda = \bigg\{ \xi \in \mathfrak{t} : \exp(2\pi \xi)= {\rm I}_N \bigg\}.
\end{equation}
Then any closed geodesic on $\T$ is the Lie exponential $\gamma(t) = \exp(2 \pi t \xi)$ for some $\xi \in \Lambda$ 
with $-\infty<t<\infty$. Note $\gamma(0)=\gamma(1)={\rm I}_N$. Now let $E_j$ (with $1 \le j \le n$) denote the block 
diagonal matrix 
\begin{equation}\label{block-Ei}
E_j = {\rm diag} (0, ..., 0, {\mathcal J}, 0, ..., 0), \qquad  
{\mathcal J} = {\mathcal J}(1) =\begin{bmatrix}
0 & -1 \\
1 & \quad 0 
\end{bmatrix}. 
\end{equation}
Then recalling the inner product (\ref{Cartan-Killing-metric-equation}) it is seen that $(E_j : 1 \le j \le n)$ forms an orthonormal basis for 
${\mathfrak t}$, hence, as any $\xi$ in this subalgebra can be expressed as $\xi= \sum_{j=1}^n a_j E_j$, upon exponentiating we have,  
\begin{align*}
\exp(2 \pi \xi) = \exp \left( \sum_{j=1}^n 2 \pi a_j E_j \right) = \prod_{j=1}^{n} \exp(2 \pi a_j E_j).
\end{align*}  
As a result 
\begin{align}
\xi \in \Lambda &\iff \exp(2 \pi \xi) = {\rm I}_N \iff \exp(2 \pi a_j E_j) = {\rm I}_N \nonumber \\
& \iff  
\left \{ \begin {array}{ll}
{\rm diag}[{\mathcal R} (2 \pi a_1), ..., {\mathcal R}(2\pi a_n)] = {\rm I}_N & \textrm{$N=2n$,} \\
{\rm diag}[{\mathcal R} (2 \pi a_1), ..., {\mathcal R}(2\pi a_n), 1] = {\rm I}_N & \textrm{$N=2n+1$,}   \nonumber 
\end{array} \right .  \\
&\iff a_j \in \Z \qquad \mbox{for all $1 \le j \le n$}.
\end{align}
This therefore identifies $\Lambda$ with $\Z^n$ which then via conjugation and translation describes all the periodic geodesics 
on $\G={\bf SO}(N)$.  Furthermore by conjugating $\Lambda$ in ${\mathfrak s}{\mathfrak o}(N)$ through ${\bf SO}(N)$ and using 
dimensional analysis it follows that every geodesic in ${\bf SO}(3)$ is periodic a conclusion that dramatically fails in ${\bf SO}(N)$ for 
$N \ge 4$. \footnote{In fact the Liouville measure of the set of periodic orbits of the geodesics flow in the co-tangent bundle $T^\star {\bf SO}(N)$ is 
zero for $N \ge 4$ ({\it see} \cite{Morris-Taheri-3} for more).} Now since the metric $\rho$ on ${\bf SO}(N)$ is bi-invariant we have
\begin{align}
\abs{\dot \gamma(t)} = \left| g_1 g \frac{d}{dt} \exp( 2 \pi t \xi) g^{-1} \right| = 
\sqrt{2 \pi^2 {\rm tr} (\xi^t \xi) } = 2\pi \abs{\xi}.  
\end{align}
Thus the length of the closed geodesic $\gamma=\gamma(t)$ is given by $l(\gamma) = 2\pi |\xi|$. 
Hence modulo translations and conjugations the number of closed geodesics whose length do not exceed 
$\sqrt{x}>0$ can be expressed as 
\begin{equation}
\mathscr{L}(x) = \sharp \bigg\{ \xi \in \Lambda : \abs{\xi} \leq \frac{\sqrt{x}}{2\pi} \bigg\} = \sum_{y\in \Z^n} \chi_r(y),
\end{equation}
where $r = \sqrt{x}/(2\pi)$ and $\chi_r$ is the characteristic function of the closed ball centred at origin with radius $r>0$. Thus the 
geodesic counting function (in the sense described) connects to the Gauss circle problem and its higher dimensional analogues -- a 
highly challenging and notoriously difficult problem in analytic number theory. Indeed $\mathscr{L}(x)$ counting 
the number of points in $\Z^n \cap \B_{\sqrt{x}/2\pi}$ has asymptotics given for suitable exponents $\delta, \zeta$ ({\it see} the table 
below) by
\begin{equation}
\mathscr{L}(x) \sim \frac{\omega_n x^\frac{n}{2}}{(2\pi)^n} + O \left( x^{\delta} ln^\zeta x \right), \quad x\nearrow \infty.
\end{equation}
%where the exponents $\delta, \zeta$ are given by the table below. 

\begin{table}[H]
\caption {The best known exponents for the remainder term in $\mathscr{L}(x)$}
% title of Table
\centering % used for centering table
\begin{tabular}{|c |c | c| c| c|} % centered columns (4 columns)
%\hline\hline %inserts double horizontal lines
\hline
& $n=2$ & $n=3$ & $n=4 $ & $n\geq 5 $ \\ [0.5ex] % inserts table
%heading
\hline % inserts single horizontal line
$(\delta,\zeta)$  & $ (131/416,0) $ & $(21/32,0)$ & $(1,2/3)$ &$ ((n-2)/2,0)$\\
\hline
Ref.  & \cite{Huxley} &  \cite{HB} & \cite{Wal} & \cite{Fricker}\\
\hline
\end{tabular}
\label{table delta} % is used to refer this table in the text
\end{table}
For $n\geq 5$ the $\delta$ in Table \ref{table delta} is sharp whereas in the cases $n\leq 3$ finding the sharp $\delta$ is still an open problem. 
A conjecture of Hardy asserts that the sharp $\delta$ for $n=2$ has the form $\delta = 1/4 + \varepsilon$ for all $\varepsilon>0$ whilst in the case 
$n=3$ the sharp $\delta$ is conjectured to be $\delta=1/2+\varepsilon$ for all $\varepsilon>0$. ({\it See} the references for more.)

Next we denote by $\mathfrak{R}$ the set of roots, by $\Delta$ the corresponding root base and by $F$ the set of 
fundamental weights. Due to the difference in the root structure of ${\bf SO}(N)$ when $N=2n$ or $N=2n+1$ we describe these two different cases separately.

\begin{itemize}
\item ${\bf SO}(2n)$: 
\begin{align*}
{\mathfrak R} &= \bigg\{ i(\pm E_k \pm E_l) : k>l \mbox{ with $1 \le l \le n$ } \bigg\}, \\
\Delta & =\bigg\{ i(E_j-E_{j+1}) : \mbox{ with $1 \le j \le n-1$} \bigg\} \bigcup \bigg\{ i(E_{n-1}+E_n)  \bigg\}, \\
F &= \left\{ \mu_j = i \sum_{k=1}^j E_k: 1\leq j \leq n-2 \right\} \bigcup 
\left\{  \mu_{n-1},\mu_n =  \frac{i}{2} \left(\sum_{k=1}^{n-1} E_k \mp E_n \right) \right\}.
\end{align*}
The lattices of weights and analytic weights for $\mathbf{SO}(2n)$ are respectively given by,
\begin{align*}
\mathscr{P} = {\rm span}_{\Z} F= \bigg\lbrace \sum_{i=1}^n b_i E_i : b_i \in \frac{1}{2}\Z \bigg\rbrace, \quad 
\mathscr{A} = \bigg\lbrace \sum_{i=1}^n b_i E_i : b_i \in \Z \bigg\rbrace.
\end{align*}
In particular the set of analytic and dominant weights (the highest weights) is given by 
\begin{align}
\mathscr{A} \cap \mathscr{C}_+ = \bigg\lbrace \sum_{i=1}^n b_i E_i : b_i \in \Z \text{  and  }  b_1\geq b_2 \geq \dots\geq \abs{b_n} \bigg\rbrace. 
\end{align} 
Note that the $\mathscr{C}_+$ denotes the positive Weyl chamber corresponding to the choice of root base $\Delta$.

\item ${\bf SO}(2n+1)$: 
\begin{align*}
{\mathfrak R} &= \bigg\{ {i}(\pm E_l \pm E_k) : k>l \mbox{ with } 1 \le l \le n \bigg\} \bigcup \bigg\{ \pm i E_k : 1 \le k \le n \bigg\}, \\
\Delta &= \bigg\{ {i}(E_j-E_{j+1}) : 1 \le j \le n-1 \bigg\} \bigcup \bigg\{ i E_n \bigg\}, \\
F &= \left\{ \mu_j= i \sum_{k=1}^{j}E_k : 1 \le j \le n-1\right\} \bigcup \left\{ \mu_n=\frac{i}{2} \sum_{k=1}^n E_k \right\}.
\end{align*}
As in the case of $\mathbf{SO}(2n)$ we have that the weights and analytic weights are given by,
\begin{align*}
\mathscr{P} = {\rm span}_{\Z} F= \bigg\lbrace \sum_{i=1}^n b_i E_i : b_i \in \frac{1}{2}\Z \bigg\rbrace,\quad 
\mathscr{A} = \bigg\lbrace \sum_{i=1}^n b_i E_i : b_i \in \Z \bigg\rbrace.
\end{align*}
However in this case we have the set of analytic and dominant weights given by 
\footnote{Therefore in the $\mathbf{SO}(2n+1)$ case, in contrast to $\mathbf{SO}(2n)$, we have $b_n\in\N_0$.}
\begin{align}
\mathscr{A} \cap \mathscr{C}_+ = \bigg\lbrace \sum_{i=1}^n b_i E_i : b_i \in \Z \text{  and  }  b_1\geq b_2 \geq \dots\geq b_n\geq 0 \bigg\rbrace. 
\end{align}

\end{itemize}

It is well known that for a compact Lie group $\G$ equipped with a bi-invariant metric $g$ the Laplace-Beltrami operator $-\Delta_g$ 
has spectrum $\Sigma= (\lambda_\mu)$ with 
\begin{align}
\lambda_{\mu} &= (\mu+\rho,\mu + \rho) - (\rho,\rho) = || \mu+\rho ||^2 - ||\rho||^2 = (\mu, \mu) + 2 (\mu, \rho)
\end{align}
where $\rho$ is the half-sum of positive roots and $\mu \in\mathscr{A}\cap\mathscr{C}_+$ ({\it cf.}, e.g., Knapp \cite{Knapp}). 
Moreover the multiplicity of the eigenvalue $\lambda_\mu$ is $\dim(\pi_\mu)^2$ where $\pi_\mu \in \hat \G$ (the unitary dual of $\G$) 
is the irreducible representation associated to $\mu\in\mathscr{A}\cap\mathscr{C}_+$, while $\dim(\pi_\mu)$ is given by 
Weyl's dimension formula. 
%
%
%The unitary dual $\widehat \G$ of $\G$ is in a one-to-one correspondence via the highest weight,with the set of analytically dominant weights 
%of the Lie algebra $\g$. Each such weight $\mu$ corresponds uniquely to an element $\lambda_\mu$ of the Casimir 
%spectrum of $\G$ ({\it cf.}, e.g., Knapp \cite{Knapp}) given by,
%\begin{align}
%\lambda_{\mu} &= (\mu+\rho,\mu + \rho) - (\rho,\rho) = || \mu+\rho ||^2 - ||\rho||^2 = (\mu, \mu) + 2 (\mu, \rho)
%\end{align}
%Here $\rho$ is the half-sum of positive roots and the inner product is that induced by the Cartan-Killing form. The Casimir spectrum describes 
%the eigenvalues of the Laplace-Beltrami operator on $\G$ and the eigenspace associated to each $\lambda=\lambda_\mu$ with multiplicity 
%given by the square of the dimension of the career space $V_\mu$, that is, $(\dim V_\mu)^2$. By Weyl's dimension formula we have 
%\begin{equation}
%\dim V_{\mu} = \prod_{\alpha \in {\mathfrak R}^+} \frac{(\alpha,\mu +\rho)}{(\alpha,\rho)} .
%\end{equation} 
%Here ${\mathfrak R}^+$ is the set of positive roots for $\g$ which are the roots given by positive combination of roots in $\Delta$. 
Restricting to $\mathbf{SO}(N)$ the eigenvalues of $-\Delta_g$ denoted 
$\Sigma=(\lambda_{\omega} : \omega \in \mathscr{A} \cap \mathscr{C}_+ )$ are given by the explicit expression
\begin{equation}
\lambda_{\omega} = \begin{cases}
 \sum_{j=1}^{n}b_j(b_j+2n-2j),  &\text{   if   } N=2n,\\
 \sum_{j=1}^{n}b_j(b_j+2n+1-2j), & \text{  if  } N=2n+1, 
\end{cases} \label{item1}
\end{equation} 
where $b=(b_1,\dots,b_n)\in \mathbb{Z}^n$ and $b_1\geq b_2\geq \dots \geq b_n\geq 0$ when $N=2n+1$ and 
$b_1 \geq b_2 \geq \dots \geq \abs{b_n}\geq 0$ when $N=2n$. Now to describe the multiplicities using 
\begin{align}
\dim (\pi_{\omega}) = \frac{\prod_{\alpha\in {\mathfrak R}^+} (\alpha,\omega + \rho) }{\prod_{\alpha\in {\mathfrak R}^+} (\alpha, \rho) },
\end{align}
we first note that 
\begin{align}
\rho = \frac{1}{2}\sum_{\alpha\in {\mathfrak R}^+} \alpha = i\begin{cases} 
\sum_{j=1}^{n-1} (n-j)E_j, & \text{   if   } N=2n, \\
\sum_{j=1}^{n} \left(n-j+1/2 \right)E_j, & \text{   if   } N=2n+1.
\end{cases}
\end{align}
Therefore the multiplicity of the eigenvalue $\lambda_{\omega}$ is given by 
\begin{equation} \label{item2}
m_{n}(x) = \dim(\pi_{\omega})^2 = 
\frac{\prod_{i<l} (x_i^2-x_l^2)^2}{\prod_{j<l}(a_j^2-a_l^2)^2}, \qquad N=2n,
\end{equation}
\begin{equation}
m_{n}(x) = \frac{\prod_i x_i^2 \prod_{i<l} (x_i^2-x_l^2)^2}{\prod_j a^2_j \prod_{j<l}(a_j^2-a_l^2)^2}, \qquad N=2n+1.
\end{equation}
Here $a_j=n-j$ when $N=2n$ and $a_j=n-j+1/2$ when $N=2n+1$ with $x=(x_1,\dots,x_n)$ given by,
\begin{align}
x_j = b_j + \begin{cases}
n-j, & \text{   if   } N=2n, \\
n-j+1/2, & \text{   if   } N=2n+1. 
\end{cases}
\end{align}

\section{Counting lattice points with polynomial multiplicities} \label{SecThree}
\setcounter {equation}{0}

Let $\Gamma \subset \R^n$ be a lattice of full rank, that is, $\Gamma = \{ \sum_{j=1}^n \ell_j v_j : \ell_j \in \Z \}$ where $v_1, ..., v_n$ 
is a fixed set of linearly independent vectors in $\R^n$. Assume $F=F(\lambda)$ is a homogenous polynomial of degree $d \ge 1$ on $\R^n$ 
assigning to each lattice point $\lambda \in \Gamma$ an associated {\it multiplicity} or weight $F(\lambda)$. The aim 
here is to describe the asymptotics of the weighted lattice point counting function
\begin{equation} \label{counting-fn-MR-equation}
\mathscr{M}(R) = \sum_{\lambda\in \Gamma} F(\lambda) \chi_{R}(\lambda) =  \sum_{\lambda\in \Gamma} F_R(\lambda). 
\end{equation}  
Here $\chi_R$ denotes the characteristic function of the closed ball in $\R^n$ centred at the origin with radius $R>0$ and $F_R=F \chi_R$. 
The approach is an adaptation of the classical argument in \cite{Hlawka} based on smoothing out the sum via convolution with a mollifier and 
then using the Poisson summation formula. To this end we consider first the "mollified sum"
\begin{align} \label{mollified-sum-equation}
\mathscr{M}_{\varepsilon}(R) &= \sum_{\lambda \in \Gamma} [F_R \star \rho_\varepsilon] (\lambda) 
=  {\rm Vol}(\Gamma^\star) \sum_{\xi \in \Gamma^\star} \widehat{F_R}(\xi) \widehat{\rho_{\varepsilon}}(\xi) \nonumber \\
&=  {\rm Vol}(\Gamma^\star) \int_{\R^n} F_R(v) \, dv  + {\rm Vol}(\Gamma^*) \sum_{\Gamma^\star \backslash \lbrace 0 \rbrace} 
\widehat{F_R}(\xi) \widehat{\rho_{\varepsilon}}(\xi),
\end{align} 
where $\Gamma^\star$ denotes the lattice dual to $\Gamma$. The focus will now be on the asymptotics of the second term on the right. 
Indeed since
\begin{equation}
\left| \sum_{\Gamma^*\backslash \lbrace 0 \rbrace} \widehat{F_R}(\xi) \widehat{\rho_{\varepsilon}}(\xi) \right| 
\leq \sum_{\Gamma^*\backslash \lbrace 0 \rbrace} \abs{\widehat{F_R}(\xi)} \abs{\widehat{\rho_{\varepsilon}}(\xi)}
\end{equation}
and by basic properties of the Fourier transform 
\begin{equation}
\widehat{F_R}(\xi) = \int_{\R^n} e^{2\pi i \xi \cdot x} F_R(x) \, dx = R^{n+d} \widehat{F_1}(R\xi)
\end{equation} 
with $|\widehat{\rho_{\varepsilon}(\xi)}| = |\widehat{\rho}(\varepsilon \xi)| \leq c_k ( 1 + \varepsilon |\xi| )^{-k}$ we can write
\begin{equation}
\left| \sum_{\Gamma^*\backslash \lbrace 0 \rbrace} \widehat{F_R}(\xi) \widehat{\rho_{\varepsilon}}(\xi) \right| 
\leq c_k R^{n+d} \sum_{\Gamma^*\backslash \lbrace 0 \rbrace}  
\frac{\abs{\widehat{F_1}(R\xi)}}{\left( 1 + \varepsilon \abs{\xi} \right)^k}. 
\end{equation}
So we now need to describe the behaviour of $\widehat{F_1}(R\xi)$ for large $R$. Towards this end and in virtue of $F$ 
being homogeneous we proceed by expressing $F$ as
\begin{equation}
F(x)= |x|^d P \left( \frac{x}{|x|} \right) = |x|^d \sum_{k=1}^{[ \frac{d}{2}]} P_k (\theta), \qquad \theta = x/|x|, 
\end{equation} 
where $P=P(\theta) $ is the restriction to the sphere $\mathbb{S}^{n-1}$ of $F$ and $P_k=P_k(\theta)$ is a spherical harmonic of degree $d-2k$ 
({\it see}, e.g., L.~Grafakos \cite{Grafakos} or Stein $\&$ Weiss \cite{Stein-Weiss}). Next let us denote %for $R>0$ the truncated functions
\begin{align}
F_1(x) %&= |x|^d P \left( \frac{x}{|x|} \right) \chi_R(x) 
%= |x|^d \chi_R \sum_{k=1}^{[ \frac{d}{2}]} P_k \left( \frac{x}{|x|} \right) \\
%&=  \chi_R \sum_{k=1}^{[ \frac{d}{2}]} |x|^{2k} P_k(x) 
=  \sum_{k=1}^{[ \frac{d}{2}]} F_{k,1}(x), \qquad F_{k,1}(x) = |x|^{2k} P_k(x)\chi_1 (x),
\end{align}  
where each $P_k$ is a solid spherical harmonic on $\R^n$ of degree $d-2k$. Clearly by the linearity of the Fourier transform we can write
\begin{equation*}
\widehat{F_1}(\xi) = \sum_{k=1}^{[ \frac{d}{2}]} \widehat{ F_{k,1}}(\xi), \qquad 
\widehat{ F_{k,1}}(\xi) = \widehat{f_{k,1}}(\abs{\xi}) P_k(\xi). 
\end{equation*}

We now momentarily focus on the quantity $\widehat{f_{k,1}}(\abs{\xi})$ % given by the weighted integral of Bessel functions 
\begin{align*}
\widehat{f_{k,1}}(\abs{\xi})  &= \frac{2\pi i^{2k-d}}{\abs{\xi}^{(n+2d-4k-2)/2}} \int_{0}^{1} s^{(n+2d)/2 } J_{(n+2d-4k-2)/2}(2\pi \abs{\xi} s) \, ds \\
&= \frac{(2\pi)^{-(n+2d)/2} i^{2k-d}}{\abs{\xi}^{(n+2d-2k)}}  \int_{0}^{2\pi \abs{\xi}} s^{(n+2d)/2}J_{(n+2d-4k-2)/2}(s) \, ds.
\end{align*}
By invoking an estimate for the weighted integral of Bessel functions that the reader can find in the Appendix ({\it see} Proposition \ref{Bessel-estimate-app}) 
we have 
\begin{align*}
\int_0^{2\pi \abs{\xi}}  s^{(n+2d)/2}J_{(n+2d-4k-2)/2}(s)ds & \leq C_k \left(2\pi |\xi| \right)^{(n+2d-1)/2} 
\end{align*}
when $\abs{\xi}>M_k$ for some $M_k\in \RR$. Therefore for $|\xi|>M_k$ we can write
\begin{align*}
\abs{\widehat{f_{k,1}}(|\xi|)} &\leq C_k |\xi|^{{-(\frac{n+1}{2} + d-2k)}}.
\end{align*}
This in turn means that,
\begin{align}
\abs{\widehat{F_1}(R\xi)} 
%&= \sum_{k=1}^{\frac{d}{2}}\abs{\widehat{f_{k,1}}(\abs{R\xi})} \abs{P_k (R\xi)}  
= \sum_{k=1}^{[\frac{d}{2}]} \abs{R\xi}^{d-2k} \abs{\widehat{f_{k,1}}(\abs{R\xi})} \left| P_k \left( \frac{\xi}{\abs{\xi}} \right) \right|  \leq c  \abs{R\xi}^{{-(n+1)/2}}
\end{align}
for $\abs{\xi}> \max_{k}(M_k)$. Therefore returning to the remainder term we get that for large enough $R$,
\begin{align}
\left| \sum_{\Gamma^*\backslash \lbrace 0 \rbrace} \widehat{F_R}(\xi) \widehat{\rho_{\varepsilon}}(\xi) \right| 
&\leq c \sum_{\Gamma^*\backslash \lbrace 0 \rbrace}  \frac{R^{d+(n-1)/2}}{\abs{\xi}^{{(n+1)/2}}\left( 1 + \varepsilon \abs{\xi} \right)^k} \nonumber \\
&\leq c \int_{\abs{\xi}\geq 1}  \frac{R^{d+(n-1)/2} \, d\xi}{|\xi|^{{(n+1)/2}}\left( 1 + \varepsilon |\xi| \right)^k} \nonumber \\
&\leq c \varepsilon^{-\frac{n-1}{2}} \int_{\RR^n}  \frac{R^{d+(n-1)/2} \, d\xi}{|\xi|^{{(n+1)/2}}\left(1+ |\xi| \right)^k} \nonumber \\
&\leq c  \frac{R^{d+(n-1)/2}}{\varepsilon^{\frac{n-1}{2}}}. 
\end{align}
As a result we can conclude that the "mollified sum" from (\ref{mollified-sum-equation}) has the asymptotic behaviour 
\begin{align}
\mathscr{M}_{\varepsilon}(R) &= {\rm Vol}(\Gamma^*) \int_{\R^n} F_R(x) \, dx  + O \left( R^{d+\frac{n-1}{2}}\varepsilon^{-\frac{n-1}{2}} \right) 
\nonumber \\
&= {\rm Vol}(\Gamma^*) R^{d+n} \int_{\R^n} F_1(x) \, dx + O \left( R^{d+\frac{n-1}{2}}\varepsilon^{-\frac{n-1}{2}} \right).
\end{align}

We next compare the mollified counting function with the original one. Towards this end we first observe that for 
$y\in \BB_{R}$ there exists some $z\in \BB_{R+\epsilon}$ such that $F_{R+\epsilon}\star \rho_{\epsilon}(y) = F_{R+\epsilon}(z)$ . Thus by combining the above 
\begin{align*}
\abs{F_R(y)-F_{R+\varepsilon}\star \rho_{\varepsilon}(y)} &= \abs{F_R(y)-F_{R+\varepsilon}(z)} \leq 2\varepsilon 
\max_{x\in \BB_{\varepsilon}(y)} \abs{\nabla F_R(x)}\leq C\varepsilon R^{d-1}, 
\end{align*}
and in addition we can obtain the similar bound for each $y\in \B_R(0)$,
\begin{align*}
\abs{F_R(y)-F_{R-\varepsilon}\star \rho_{\varepsilon}(y)}
% &= \abs{F_R(y)-F_{R-\varepsilon}(z)} \leq 2\varepsilon \max_{x\in \BB_{\varepsilon}(y)} \abs{\nabla F_R(x)} 
\leq C \varepsilon R^{d-1}.
\end{align*}
Therefore it follows that for each $y\in\B_R(0)$ we have,
\begin{align*}
F_{R-\varepsilon}\star \rho_{\varepsilon}(y) - C\varepsilon R^{d-1} \leq F_R(y) 
\leq F_{R+\varepsilon} \star \rho_{\varepsilon}(y) + C \epsilon R^{d-1}
\end{align*}
thus giving
\begin{equation} \label{prev-obt-equation}
\mathscr{M}_{\varepsilon}(R-\varepsilon) - C\varepsilon R^{d-1} \sum_{\lambda \in \Gamma} \chi_{R}(\lambda) \leq \mathscr{M}(R) 
\leq \mathscr{M}_{\varepsilon}(R+\varepsilon) + C \epsilon R^{d-1}\sum_{ \lambda \in \Gamma} \chi_{R}(\lambda).
\end{equation}

Next referring to the original lattice $\Gamma$ we can define with the aid of the basis vectors $v_1, ..., v_n$ another lattice 
\begin{equation}
\Omega 
%= {\rm span}_{\Z} \bigg\{ \frac{a v_1}{\norm{v_1}}, \dots, \frac{a v_n}{\norm{v_n}} \bigg\} 
= \bigg\{ \omega = \sum_{j=1}^n \ell_j \frac{a v_j}{||v_j||} : \ell_j \in \Z \bigg\}
\end{equation} 
where $a = \min_j \lbrace \norm{v_j} \rbrace>0$. Then using the bound
\begin{equation}
\sum_{\Gamma} \chi_R \leq \sum_{\Omega} \chi_R = \sum_{\Z^n} \chi_{\frac{R}{a}} \leq \frac{R^n}{a^n},
\end{equation}
we can rewrite (\ref{prev-obt-equation}) as
\begin{align*}
\mathscr{M}_{\varepsilon}(R-\varepsilon) - C\varepsilon R^{d+n-1} \leq \mathscr{M}(R) \leq \mathscr{M}_{\varepsilon}(R+\varepsilon) 
+ C \epsilon R^{d+n-1}.
\end{align*}
Therefore the previously obtained bounds for $\mathscr{M}_\varepsilon$ result in
\begin{align*}
\mathscr{M}(R) = \left( {\rm Vol}(\Gamma^*) \int_{\R^n} F_1(x) \, dx  \right) R^{d+n} + O\left( R^{d+\frac{n-1}{2}}\varepsilon^{-\frac{n-1}{2}} 
+ \varepsilon R^{d+n-1}\right) 
\end{align*}
as $R\nearrow \infty$. Noting that the remainder term is optimised when $\varepsilon=R^{-\frac{n-1}{n+1}}$ leads to the following 
conclusion.

\begin{theorem}\label{radial} Let $F$ be a homogeneous polynomial of degree $d \ge 1$ on $\R^n$ and let 
$\Gamma \subset \R^n$ be a lattice of full rank. Consider the weighted counting function $\mathscr{M}=\mathscr{M}(R)$ 
defined for $R>0$ by \eqref{counting-fn-MR-equation}.
%\begin{equation}
%\mathscr{M}(R) = \sum_{\lambda\in \Gamma} F(\lambda) \chi_{R}(\lambda). 
%\end{equation}
Then
\begin{equation} \label{asympt}
\mathscr{M}(R) = \left( {\rm Vol}(\Gamma^*)  \int_{\B^n_1} F(x) \, dx  \right) R^{d+n} + O\left( R^{d+n -\frac{2}{n+1}}\right), \qquad 
R \nearrow \infty.
\end{equation}
\end{theorem}

Note that repeating the above proof for the shifted counting function $\mathscr{M}^h(R)=\sum_{\lambda \in \Gamma} F_R(\lambda+h)$ 
with $\mathscr{M}^h_\varepsilon(R) = \sum_{\lambda \in \Gamma} [F_R\ast \rho_\varepsilon](\lambda+h)$ results in the exact same 
asymptotics \eqref{asympt} for $\mathscr{M}^h$.
%\begin{equation*}
%\mathscr{M}^h(R) = \left( {\rm Vol}(\Gamma^*)  \int_{\B^n_1} F(v) \, dv  \right) R^{d+n} + O\left( R^{d+n -\frac{2n}{n+1}}\right), 
%\qquad R \nearrow \infty.
%\end{equation*}  
\footnote{This is a consequence of the identities $\widehat {f(\cdot +h)}(\xi) = \widehat{f}(\xi) e^{2\pi i h \cdot \xi}$, $\abs{\widehat {f(\cdot +h)}(\xi)} 
= \abs{\widehat{f}(\xi)}$.} We use this remark later on.

\section{Improved asymptotics for $\mathscr{N}(\lambda; \G)$ when ${\mathbb G}={\bf SO}(N)$, $\mathbf{SU}(N)$, $\mathbf{U}(N)$ and ${\bf Spin}(N)$} 
\label{SecFour}
\setcounter{equation}{0}

This section is devoted to the analysis of the asymptotics of the spectral counting function $\mathscr{N}(\lambda; \G)$ as $\lambda \nearrow \infty$ 
when $\G$ is one of the special orthogonal or unitary groups in the title. Here, the calculations in light of what has been obtained so far is explicit and 
the main question is the behaviour of the remainder term and whether it agrees with the Avakumovic-H\"ormander sharp form or if there is an improvement. 
Notice that ${\bf SO}(2) \cong {\mathbb S}^1$ and ${\bf SO}(3) \cong {\mathbb S}^3 / \{\pm 1\} \cong {\bf P}({\mathbb R}^3)$, the real projective space, and 
so in view of the periodicity of the geodesic flow (or direct calculations) we do not expect any improvements. However remarkably things change 
sharply as soon as we pass to the higher dimensional cases ${\bf SO}(N)$ (with $N \ge 4$). 
Indeed from earlier discussions we know, using \eqref{item1}-\eqref{item2}, that the spectral counting function for ${\bf SO}(N)$ is given by
\footnote{Below we shall be using the notation of $\mathscr{A}_\rho = \mathscr{A} + \rho$. }
\begin{align}\label{counting1}
\mathscr{N}(\lambda) = \sum_{\omega\in \mathscr{A} \cap {\mathscr{C}}_+} m_n(\omega+\rho) \chi_R(\omega+\rho) 
= \sum_{x\in \mathscr{A}_\rho \cap \mathring{\mathscr{C}}_+} m_n(x) \chi_R(x),
\end{align} 
where $\chi_R$ is the characteristic function of the closed ball with $R=\sqrt{\lambda + \norm{\rho}^2}$ centred at the origin and $m_n$ 
is the multiplicity function that is explicitly by \eqref{item2}. In \eqref{counting1} we have also let $x=\omega+\rho$ and used the fact that 
$x \in \mathring{\mathscr{C}}_+$ for $\omega \in \mathscr{A} \cap \mathscr{C}_+$. An easy inspection show that on 
$\partial \mathscr{C}_+$ we have $m_n(x)=0$ therefore we can rewrite $\mathscr{N}(\lambda)$ as 
\begin{equation}
\mathscr{N}(\lambda) = \sum_{x\in \mathscr{A}_\rho \cap \mathscr{C}_+} m_n(x) \chi_{R}(x).
\end{equation}
Notice that the multiplicity function $m_n$ is invariant under any permutation of $(x_1,\cdots,x_n)$ in either case. In addition $m_n$ is also invariant under 
any change in sign of $n-1$ of the $x_i$'s when $N=2n$ and invariant under any change of sign of all the $x_i$'s when $N=2n+1$. Thus $m_n$ is invariant 
under the Weyl group $W$ given by 
\begin{align}
W = \begin{cases}
\Z_2^{n-1} \rtimes S_n,  &\text{   if   } N=2n, \\
\Z^n_2 \rtimes S_n,  &\text{    if    } N=2n+1.   
\end{cases}
\end{align}  
We now take advantage of the action of the Weyl group on the set of weights $\mathscr{A}_\rho \cap \mathscr{C}_+$ to extend $\mathscr{N}(\lambda)$ 
to the full set of weights $\mathscr{A}_\rho$. \footnote{It can be easily checked that the Weyl group $W$ maps $\mathscr{A}_{\rho}$ to itself since for 
each $w \in W$ we have that $w\cdot \rho = \rho - \alpha$ for some $\alpha \in \mathfrak{R}\subset \mathscr{A}$. Then as $W \cdot \mathscr{A} 
= \mathscr{A}$ we clearly have that for any $w\in W$, $w\cdot(\mu + \rho) \in \mathscr{A}_{\rho}$. } Indeed as the Weyl group acts simply transitively on 
the interior of the Weyl chambers (which means that the interior of any Weyl chamber is mapped onto the interior of any other chamber in a bijective 
manner) we can write 
%\footnote{The simple transitivity of the action of the Weyl group on the Weyl chambers is also true in general for compact Lie groups.} 
%This set of analytic weights can be identified with $\Z^n$ and so
\begin{align*} %\label{NandM}
\mathscr{N}(\lambda) &= \frac{1}{\abs{W}}\sum_{x \in \mathscr{A}_\rho} m_n(x)\chi_{R}(x) = \frac{1}{\abs{W}} 
\sum_{\omega\in \mathscr{A}}m_n(\omega +\rho)\chi_{R}(\omega+\rho) = \mathscr{M}^\rho(R). 
\end{align*} 
Now $m_n$ is a homogeneous polynomial of degree $2l=d-n$ where $d=\dim[{\mathbf{SO}(N)}]$ and $n={\rm Rank}[\mathbf{SO}(N)]$. 
Then since $\mathscr{A}$ can be identified with $\Z^n$ we have that ${\rm Vol}(\mathscr{A}^*)=1$ and then from Theorem \ref{radial} we 
deduce that
\begin{align*}
\mathscr{M}^\rho(R) = \frac{R^d}{\abs{W}} \int_{\RR^n} m_n(x) \chi_1(x) dx  + O\left( R^{d - 1-\frac{n-1}{n+1}}\right), 
\end{align*} 
as $R \nearrow \infty$. Hence in view of ${\mathscr{N}(\lambda)}=\mathscr{M}(R)$, when $R=\sqrt{\lambda + \norm{\rho}^2}$, we obtain
\begin{align*}
\mathscr{N}(\lambda) =  \frac{\lambda^{d/2}}{\abs{W}} \int_{\RR^n} m_n(x) \chi_1(x) dx  + O\left( \lambda^{{\frac{d - 1}{2}-\frac{n-1}{2(n+1)}}}\right). 
\end{align*} 
The leading term can be evaluated to be 
\begin{align*}
\frac{\lambda^{d/2}}{\abs{W}}\int_{\BB_1}m_n(x) dx  = \frac{\omega_d {\rm Vol}_g({\bf SO}(N))}{(2\pi)^d}\lambda^{d/2}.
\end{align*}
Subsequently it follows that  
\begin{equation}
\mathscr{N}(\lambda) = \frac{\omega_d {\rm Vol}_g({\bf SO}(N))}{(2\pi)^d}\lambda^{\frac{d}{2}}  + O(\lambda^{\alpha}), \quad \lambda \nearrow \infty,
\end{equation}
where the exponent $\alpha$ in the remainder term, by making use of $d=(N^2-N)/2$, is seen to be 
\begin{align}
\alpha = \frac{1}{2(n+1)}\begin{cases}
{{2n^3+n^2-3n} } & \text{    if    } N=2n, \\
{{2n^3+3n^2-n}} & \text{    if    } N=2n+1.
\end{cases} \label{alpha}
\end{align}
This in particular confirms that the remainder term in Weyl's law is not sharp for the compact Lie group ${\bf SO}(N)$. 
In summary we have proved the following result.

\begin{theorem} The spectral counting function $\mathscr{N}=\mathscr{N}(\lambda; \mathbf{SO}(N)$ of the Laplace-Beltrami operator with 
$n = {\rm rank}[\mathbf{SO}(N)] \ge 2$ has the asymptotics $($$\lambda \nearrow \infty$$)$ 
\begin{equation}
\mathscr{N}[\lambda; {\bf SO}(N)] = \frac{{\rm Vol}({\bf SO}(N))\omega_{d}}{(2\pi)^{d}} \lambda^{\frac{d}{2}}  
+ O(\lambda^{\alpha}),
\end{equation}
with $d=\dim(\mathbf{SO}(N))$ and $\alpha$ given by \eqref{alpha}. 
%$\frac{n(2n+3)(n-1)}{2(n+1)}$ if $N=2n$ and $\alpha = \frac{(n^2+3n+1)(2n-1)}{2(n+1)}$ when $N=2n+1$.
\end{theorem}
\begin{table}[H]
\caption {$\mathbf{SO}(N)$ and ${\bf Spin}(N)$}
% title of Table	
\centering % used for centering table
\begin{tabular}{|c |c |c|} % centered columns (4 columns)
%\hline\hline %inserts double horizontal lines
\hline
$\G$ & $\mathbf{SO}(N)$ & ${\bf Spin}(N)$\\ [0.5ex] % inserts table
%heading
\hline % inserts single horizontal line
$n=\mbox{rank}(\G)$  & %\multicolumn{2}{|c|}{}
$[N/2]$ & $[N/2]$\\
\hline
$d=\dim(\G)$ & %\multicolumn{2}{|c|}{}
$N(N-1)/2$ & $N(N-1)/2$  \\ % inserting body of the table
%		\hline
%		$a_j=\langle E_j,\rho \rangle$  & $ n-j $ & $ n-j+1/2 $   \\
\hline
$Q=\prod_{\alpha\in R^+} (\alpha,\rho) $  &\multicolumn{2}{|c|}{ $  (2^{-n}N!!)^{N-2n}\prod_{j=1}^{n-1} j! \prod_{j<k} (N-j-k) $} \\
\hline
${\rm Vol}(\G) \times Q$  & $ (2\pi)^{{N(N-1)}/4+n/2} $ & $2^n(2\pi)^{{N(N-1)}/4+n/2} $ \\
\hline
\end{tabular}
\label{table vol 1} % is used to refer this table in the text
\end{table}

We now present the analogous analysis and result for the unitary and special unitary groups ${\bf U}(N)$ and ${\bf SU}(N)$ respectively. 
Firstly note that the spectrum of the Laplace-Beltrami on ${\bf U}(N)$ and ${\bf SU}(N)$ is given by the following
\begin{align}
\lambda_{\omega} & = \sum_{j=1}^N \left[(b_j-j+(N+1)/2)^2-((N+1)/2-j)^2\right],
%\langle X_{\omega} + Y_{\rho}, X_{\omega} + Y_{\rho} \rangle - \langle Y_{\rho}, Y_{\rho} \rangle
\end{align}
where $b_1\geq b_2\geq \dots \geq b_N $ with $b_j\in\Z$ for $\mathbf{U}(N)$ or $b_j\in\Z+b_N$ for $1\leq j \leq n-1$ and $b_N \in \Z/N$ whilst \eqref{b_N} holds for $\mathbf{SU}(N)$. 
The multiplicity of these eigenvalues is given by,
\begin{align}
m_N(x) =  \prod_{1\leq j<k\leq N} \frac{\left( x_j-x_k \right)^2 }{\left(k-j\right)^2},
\end{align}
where $x=(x_1,\dots,x_N)$ and $x_j = b_j-j+(N+1)/2 $. Note that in the case of $\mathbf{SU}(N)$ the 
\begin{equation}
b_N= -\sum_{j=1}^{N-1}b_j, \label{b_N}
\end{equation}	
which therefore means that the eigenvalues and corresponding multiplicity function only depend on $b_1,\dots,b_{N-1}$ 
and $x_1,\dots,x_{N-1}$ respectively. Now following the arguments of $\mathbf{SO}(N)$ we can prove the following (the 
proof of which we shall omit due to the similarity with $\mathbf{SO}(N)$).

\begin{theorem} The spectral counting function $\mathscr{N}(\lambda)=\mathscr{N}(\lambda; \G)$ of the Laplace-Beltrami operator on the 
unitary group $\G={\bf U}(N)$ with $N \ge 2$ and $d=\dim(\mathbf{U}(N))=N^2$ has the asymptotics $($$\lambda \nearrow \infty$$)$
\begin{equation}
\mathscr{N}[\lambda; {\bf U}(N)] = \frac{{\rm Vol}({\bf U}(N))\omega_{d}}{(2\pi)^{d}} \lambda^{\frac{d}{2}} 
+ O \left( \lambda^{\frac{N(N+2)(N-1)}{2(N+1)}} \right).
\end{equation}
Likewise in the case of the special unitary group $\G={\bf SU}(N)$ with $N \ge 2$ and $d=\dim(\mathbf{SU}(N))=N^2-1$ we have that
\begin{equation}
\mathscr{N}[\lambda; {\bf SU}(N)] = \frac{{\rm Vol}({\bf SU}(N))\omega_{d}}{(2\pi)^{d}} \lambda^{\frac{d}{2}} 
+ O \left( \lambda^{\frac{N^3-3N+2}{2N}} \right).
\end{equation} 
\end{theorem}

Note that the metric is the one arising from the inner product $(X,Y)={\rm tr}\,(X^\star Y)$ and is bi-invariant. By inspection 
for ${\bf U}(N)$ when $N \ge 2$ and for ${\bf SU}(N)$ when $N \ge 3$ the remainder term in Weyl's law 
\eqref{Weyl-law-remainder-general-equation} is not sharp whilst evidently outside this range the 
geodesic flow on the group is periodic.

\begin{table}[H]
\caption {${\bf SU}(N)$ and ${\bf U}(N)$}
% title of Table
\centering % used for centering table
\begin{tabular}{|c |c | c|} % centered columns (4 columns)
%\hline\hline %inserts double horizontal lines
\hline
$\G$ & ${\bf SU}(N)$ & ${\bf U}(N)$ \\ [0.5ex] % inserts table
%heading
\hline % inserts single horizontal line
$n=\mbox{rank}(\G)$  & $N-1$ & $N$\\
\hline
$d=\dim(\G)$ &  $N^2-1 $ & $N^2$ \\ % inserting body of the table
%\hline
%$2l=\abs{R}$   & $ N^2-N $ & $N^2-N $ & $\lceil ({N^2-2N})/{2} \rceil $   \\
\hline
$Q=\prod_{\alpha\in R^+} (\alpha,\rho) $  & $\prod_{j=1}^{N-1}j!$ 
& $\prod_{j=1}^{N-1}j!$ \\
\hline
${\rm Vol}(\G) \times Q$  & ${N(2\pi)^{(N+2)(N-1)/2}}$ & $  {(2\pi)^{N(N+1)/2}}$\\
\hline
\end{tabular}
\label{table vol 2} % is used to refer this table in the text
\end{table}

Let us end the section by studying the asymptotics of $\mathscr{N}(\lambda; \G)$ for when ${\mathbb G}={\bf Spin}(N)$ is the universal 
cover of ${\bf SO}(N)$. In virtue of $\pi_1[{\bf Spin}(N)] \cong 0$ there is a one-to-one correspondence between the [complex] irreducible 
representations of ${\bf Spin}(N)$ and those of its Lie algebra ${\bf {\mathfrak s}{\mathfrak p}{\mathfrak i}{\mathfrak n}}(N)$. Moreover as 
${\bf Spin}(N)$ is a double cover of ${\bf SO}(N)$ the irreducible representations of the latter are only "half" the total of the former and 
hence of the Lie algebras $\mathfrak{spin}(N) \cong {\bf {\mathfrak so}}(N)$. Despite this we have the following conclusion for 
${\bf Spin}(N)$ based on what was obtained previously for ${\bf SO}(N)$. 
% for the universal covering group ${\bf Spin}(N)$ despite there being more dominant 
%weights and hence more unitary irreducible representations associated to the former. 

\begin{theorem}\label{spin}
Consider the universal covering group $\G={\bf Spin}(N)$ of ${\bf SO}(N)$ with $N \ge 3$. Then the spectral counting function 
$\mathscr{N}=\mathscr{N}(\lambda; \G)$ of the Laplace-Beltrami operator has the asymptotics 
\begin{equation}
\mathscr{N}(\lambda; \G) = \frac{\omega_d {\rm Vol}_g(\G)}{(2\pi)^d} \lambda^{\frac{d}{2}} 
+ O \left( \lambda^{\alpha} \right), \qquad \lambda \nearrow \infty, 
\end{equation}
where $d=\dim {\bf Spin}(N)=\dim {\bf SO}(N)=N(N-1)/2$ and $\alpha$ is given by \eqref{alpha}.
\end{theorem}

\pf As eluded to in the discussion prior to the theorem in virtue of $\G={\bf Spin}(N)$ being the universal covering group of ${\bf SO}(N)$ 
the two share the same root system, and the set of analytically dominant weights $\mathscr{A} \cap \mathscr{C}_+$ for ${\bf Spin}(N)$ are,
\begin{equation}
\mathscr{A} \cap \mathscr{C}_+  = \bigg\{ \xi + \varepsilon (1,\cdots,1) : \xi \in \mathscr{A}_{{\bf SO}(N)} \cap \mathscr{C}_+, \,\, \varepsilon=0 \text{  or  } 
\varepsilon=\frac{1}{2} \bigg\} = \mathscr{P}_{{\bf SO}(N)}\cap \mathscr{C}_+. 
\end{equation}

%This and the fact that 
The root systems of ${\bf Spin}(N)$ and ${\bf SO}(N)$ being the same implies that the multiplicity function $m_n$ for the two are the 
same homogenous polynomial.
% as for the latter but clearly now extended to the additional weights. 
In particular for each $y\in \mathscr{A} \cap \mathscr{C}_+$ we can write $y=x/2$ for some $x \in \mathscr{A}_{{\bf SO}(N)} \cap \mathscr{C}_+$ 
and so by homogeneity $m_n(y)= 2^{-2l} m_n(x)$ (note that $d=n+2l$ where $n$ is the rank). Therefore the 
counting function of ${\bf Spin}(N)$ relates to the counting function of ${\bf SO}(N)$ by rescaling, specifically,  
\begin{align}
\mathscr{M}_{{\bf Spin}(N)}(R) &= 2^{-2l} \mathscr{M}_{{\bf SO}(N)}(2R) \nonumber \\
&= 2^n \frac{\omega_d {\rm Vol}_g({\bf SO}(N))}{(2\pi)^d}  R^{d} + O \left( R^{d-1-\frac{n-1}{n+1}} \right).  
\end{align}

Now we reach the desired conclusion by invoking the relation $R= \sqrt{\lambda + \norm{\rho}^2}$ and hence obtaining
\begin{equation} \label{scf-Spin-SO-equation}
\mathscr{N}(\lambda) = 2^n \frac{\omega_d {\rm Vol}_g({\bf {\bf SO}}(N))}{(2\pi)^d} \lambda^{\frac{d}{2}} 
+ O \left( \lambda^{\frac{d-1}{2}- \frac{1}{2} \frac{n-1}{n+1}} \right)  
\end{equation}
and making use of the relation $2^n {\rm Vol}({\bf SO}(N)) = {\rm Vol}({\bf Spin}(N))$. \hfill $\square$ 
\\[1 mm]

Regarding the last relation in the above proof we
% see Table \ref{table vol 2} below. Indeed 
note that the volume of any compact Lie group $\G$ is given by
\begin{equation}
{\rm Vol}(\G) = \frac{(2\pi)^{n+l}}{{\rm Vol}(\mathscr{A}) \times Q}, \qquad Q= \prod_{\alpha \in {\mathfrak R}^+} (\alpha,\rho), 
\end{equation} 
where ${\rm Vol}(\mathscr{A})$ is the volume of the fundamental domain in the lattice of analytical weights. Then as 
%$\dim {\bf SO}(N) = \dim {\bf Spin}(N) = n+2l$ 
%and both 
${\bf Spin}(N)$ and ${\bf SO}(N)$ share the same root system and ${\rm rank}(\G)=n$ 
%and so $\prod_{\alpha \in R^+} (\alpha,\rho)$ 
the above claim follows upon noting that ${\rm Vol}(\mathscr{A}_{\mathbf{SO}(N)})=1$ and ${\rm Vol}(\mathscr{A}_{{\bf Spin}(N)})= 2^{-n}$.
% The subscripts $1$ and $2$ refer to ${\bf SO}(N)$ and ${\bf Spin}(N)$ respectively. 

\section{Weyl's Law for the Orthogonal and Unitary groups: A sharp result} \label{SecFive}
\setcounter {equation}{0}

In this final section we present a result on the sharp asymptotics of the remainder term for the spectral counting function of the orthogonal and unitary 
groups as in the previous sections. Here we assume for technical reasons that the rank of the group is strictly greater than four.

\begin{theorem}\label{sharpthm}
Let $\G$ denote one of the unitary, orthogonal or spinor groups as above. Then provided that $n={\rm rank}(\G)\geq 5$ we have 
\begin{align}\label{sharp}
\mathscr{N}(\lambda; \G) = \frac{\omega_{d} {\rm Vol}(\G)}{(2\pi)^{d}} \lambda^{\frac{d}{2}} + O(\lambda^{\frac{d-2}{2}}), \qquad \lambda \nearrow \infty, 
\end{align} 
where as before $d=\dim(\G)$. 
\end{theorem}   
The principle idea of the proof is to approximate the counting function $\mathscr{N}(\lambda)$ by an alternative one with radial weight which is easier 
to tame. Recall that the multiplicity functions, given earlier in the paper, are homogeneous polynomials of even degree and as such can be written as 
$m_n(x) = \abs{x}^{2m}P(x/\abs{x})$. This homogeneity permits the forthcoming description of $\mathscr{N}(\lambda)$ where $n={\rm rank}(\G)$ and 
$r_n(k) =  \abs{\lbrace \omega\in\Z^n: \abs{\omega}^2=k \rbrace}$. Indeed 
\begin{align}
\mathscr{N}(\lambda) = \frac{1}{\abs{W}} \sum_{k=1}^{R^2} k^m r_n(k) \left(\frac{1}{r_n(k)} \sum_{\theta_j \in H_k} P(\theta_j) \right), \label{ave}
\end{align}
where $H_k = \lbrace \theta_j = x_j/|x_j| : x_j\in \Z^n \text{   and   } \abs{x_j}^2=k \rbrace$. The form \eqref{ave} is suggestive towards a natural 
approximation by
\begin{align} \label{MR}
\EE(R) =   \sum_{\omega\in\Z^n}\abs{\omega}^{2m} \chi_{R}(\omega) = \sum_{k=1}^{R^2} k^m r_n(k).
\end{align}
It is clear that any sensible approximation of $\mathscr{N}(\lambda)$ via $\mathscr{E}(R)$ necessitates that the averaged sum in \eqref{ave} 
converges. To this end we recall the uniformly distributed nature of $\Z^n$ projected onto the unit sphere $\Sp^{n-1}$, giving,
\begin{align}
\frac{1}{r_n(k)} \sum_{\theta_j \in H_k} F(\theta_j) \rightarrow \frac{1}{\abs{\Sp^{n-1}}}\int_{\Sp^{n-1}} F(\theta) d \mathcal{H}^{n-1}(\theta),
\end{align}
as $k\rightarrow \infty$ for any $F\in {\bf C}(\Sp^{n-1})$. Through the work of C. Pommerenke $\cite{Pom}$ and 
A.V. Malyshev $\cite{Malysev}$, it is even known that for $F\in {\bf C}^{2m}(\Sp^{n-1})$ we have the quantitative estimate,
\begin{align*}
\biggl |\frac{1}{r_n(k)} \sum_{\theta_j \in H_k} F(\theta_j) - \frac{1}{\abs{\Sp^{n-1}}}\int_{\Sp^{n-1}} F(\theta) d \mathcal{H}^{n-1}(\theta)\biggl| 
\leq \frac{c(m,n)}{k^{\frac{n-1}{4}}} \norm{\Delta^m_{\Sp^{n-1}} F}_{L^1(\Sp^{n-1})},
\end{align*}
provided $m>n-1$. \footnote{For more on approximation results of this nature the reader is referred to $\cite{Freeden}$ pp.~187-192.} 
Therefore as $P(\theta)$ is the restriction to $\Sp^{n-1}$ of a homogeneous polynomial we have that,
\begin{align*}
\biggl |\frac{1}{r_n(k)} \sum_{\theta_j \in H_k} P(\theta_j) - \kappa \biggl|\, \leq \frac{ c(m,n,P)}{k^{\frac{n-1}{4}}},
\end{align*}
where $\kappa = {\abs{\Sp^{n-1}}}^{-1}\int_{\Sp^{n-1}} P(\theta) d \mathcal{H}^{n-1}(\theta)$. Consequently
\begin{align}
\abs{\NN(\lambda) - \kappa {\abs{W}}^{-1}\, \EE(R)} \leq \frac{1}{\abs{W}} \sum_{k=0}^{R^2}k^{m} r_n(k) \frac{ c(m,n,P)}{k^{\frac{n-1}{4}}}.
\end{align}
Combining this with $r_n(k) = O(k^{\frac{n-2}{2}})$ for $n \geq 5$, {\it see} \cite{Freeden}, grants that,
\begin{align}
\abs{\NN(\lambda) - \kappa \abs{W}^{-1}\EE(R)}  = O\left( R^{2m+\frac{n+1}{2}} \right). \label{approx}
\end{align}  
Therefore once we have a result analogous to Theorem \ref{sharpthm} for $\EE(R)$ we shall obtain \eqref{sharp} as a simple repercussion 
of the above approximation estimate. 
\begin{theorem}\label{SharpM}
Let $\EE=\EE(R)$ be as in \eqref{MR}. Then provided $n\geq 5$ we have the asymptotics
\begin{align}
\EE(R) = \frac{{\rm Vol}(\Sp^{n-1})}{2m+n} R^{2m+n} + O(R^{2m+n-2}), \qquad R \nearrow \infty. \label{MRThm}
\end{align}
\end{theorem}

\begin{proof}
The proof of this result is an adaptation of the classical lattice point counting argument with constant weight, i.e., $m=0$. Indeed the explicit form of $r_4(k)$, i.e., 
the Jacobi sum of four square formula, gives a weaker result for $\EE(R)$ on $\Z^4$. More precisely,
\begin{align}
\EE_4(R) = & \int_{\B_R} \abs{x}^{2m} \, dx+ \nonumber \\
&  \frac{16/R^2}{{\rm Vol}(\Sp^{3})}\int_{\B_R} \abs{x}^{2m} dx \cdot \left[ D(R^2) - D({R^2}/{4}) \right]  + O(R^{2m+2}), \label{useful2}
\end{align}
where $D(t)= \sum_{k\leq t} k^{-1} \psi (t/k)$ with $\psi(t/k)=t/k - [ t/k ] -1/2$. To prove \eqref{useful2} one can firstly show that ({\it cf.} \cite{Fricker} pp.~34-5)
 \footnote{In \cite{Fricker} the identity \eqref{MS} is proved when $m=0$. However the case $m>0$ is similar modulo suitable adjustments to essentially account 
for $m \ne 0$. Therefore to avoid repetition with an existing text we shall simply refer the reader to \cite{Fricker}.}
\begin{align}
\EE_4(R)  =8 S_m(R^2) - 4^m \times 32\, S_m\left(\frac{R^2}{4}\right), \label{MS}
\end{align}  
where $S_m(t) = \sum_{k=1}^t k^m \sigma(k)$ and $\sigma(k)$ is the classical divisor function, i.e., $\sigma (k) = \sum_{d|k}d$. Notice that any pair of integers 
$(d,j)$ such that $1\leq d \leq t$ and $1\leq j \leq [ t/d ] = t/k - \psi(t/k) - 1/2 =\zeta(d,t)$ are divisors of $d\cdot j$ where $1\leq d \cdot j  \leq t$. Hence,
\begin{align*}
S_m(t) &= \sum_{k\leq t} k^m \sum_{d|k} d =  \sum_{k\leq t} \sum_{j\leq \zeta(k,t)} (jk)^m j = \sum_{k\leq t} \frac{k^m}{m+2} \sum_{j=0}^{m+2} c_{m,j} \zeta(k,n)^{m+2-j},
\end{align*}
where the last equality comes from Faulhaber's formula with $c_{m,j}=(-1)^j {m+2 \choose j} B_j$ and $B_j$ are the Bernoulli numbers. Moreover,
\begin{align}
\zeta(k,t)^m = \left(\frac{t}{k}\right)^m - \left(\frac{t}{k}\right)^{m-1}\left(\psi(\frac{t}{k})+1/2\right) +  O\left( \frac{t^{m-2}}{k^{m-2}} \right).
\end{align}  
As a result,
\begin{align}
S_m(t) &= \sum_{k\leq t} \frac{k^m}{m+2} \sum_{j=0}^{m+2} c_{m,j} \left(\frac{t}{k}\right)^{m+1-j}\zeta(k,t) + O(t^{m+1}) \nonumber\\
& = \frac{t^m}{m+2}\sum_{k\leq t}\left(\frac{t}{k}+ c_{m,1}\right)\zeta(k,t) + O(t^{m+1}).  \label{Sm2}
\end{align}
Now $\sum_{k\leq t} t/k \zeta(k,t) = t^2\sum_{k\leq t} k^{-2} - tD(t) - t/2 \sum_{k\leq t} k^{-1} $ 
and $\sum_{k\leq t} \zeta(k,t) = t\sum_{k\leq t}k^{-1} + O(t)$ which permits \eqref{Sm2} to be rewritten as (with $c=c_{m,1}-2^{-1}$)
\begin{align}\label{Sm3}
S_m(t) & = \frac{t^{m+2}}{m+2}\sum_{k\leq t}\frac{1}{k^2} +  \frac{t^{m+1}}{m+2}\left[ c\sum_{k\leq t} k^{-1} -  D(t) \right] + O(t^{m+1}).  
\end{align}
Next substituting \eqref{Sm3} into \eqref{MS} with $\sum_{k\leq R^2}k^{-1}-\sum_{k\leq R^2/4}k^{-1}=O(1)$ and using 
$1+ \cdots + 1/N^2 = \pi^2/6 + O(N^{-1})$ gives  
\begin{align}
\EE_4(R) = \frac{ R^{2m+4} \pi^2}{m+2} - \frac{8R^{2m+2}}{m+2} \left[ D(R^2) - D({R^2}/{4}) \right] + O(R^{2m+2}). \label{MRfor4dim}
\end{align}
Thus we have proved \eqref{useful2} upon identifying the coefficients in \eqref{MRfor4dim} with the integrals given in \eqref{useful2}. 
Now \eqref{MRThm} is obtained by firstly writing,
\begin{align}
\EE_5(R) & = \sum_{\omega \in \Z^5} \abs{\omega}^{2m} \chi_R(\omega) = \sum_{j=-R}^{R} \EE_4^{f_j}(\sqrt{R^2-j^2}). \label{basecase}
\end{align}
Here $\EE_4^{f_j}(t)$ denotes the counting function on $\Z^4$ with weight $f(\omega,j)= (\abs{\omega}^2+j^2)^m$ which is a sum 
of radial weights. Hence \eqref{useful2} in \eqref{basecase} produces,
\begin{align}
\EE_5(R) & = \sum_{j=-R}^{R} \int_{\B_{r(R^2,j)}} f(x,j)\, dx - 16 \left[ H(R^2) - L(R^2)\right] + O(R^{2m+3}), \label{E5}
\end{align}
where we have defined $r(t,j)=\sqrt{t-j^2}$ and \footnote{The remainder term here accounts for $j=0$ which is 
$t^{m+1}D(t) = O(t^{m+1}\ln t)$ since $D(t)=O(\ln t)$, however, in what follows we show that $H(t),L(t)=O(t^{m+3/2})$ 
and therefore we can omit this additional remainder as it will be adsorbed into this asymptotics.},
\begin{align}
H(t) &= \sum_{j=1}^{[ \sqrt{t}]} F(t,j)D(t-j^2) + O(t^{m+1}\ln t), \\
L(t) &= \sum_{j=1}^{[ \sqrt{t}]} F(t,j)D((t-j^2)/4)  +  O(t^{m+1}\ln t),
\end{align}
and $F(t,j)= 2{r(t,j)^{-2}} {\rm Vol}(\Sp^3)^{-1}\int_{\B_{r(t,j)}} f(x,j) \, dx $. Moreover we have 
\begin{align}
H(t) & =  \sum_{j=1}^{[ \sqrt{t} ] -1} \left[F(t,j)-F(t,j+1) \right]M(t,j) + F(t,[ \sqrt{t} ]  )M([t, \sqrt{t}] ) 
\end{align}
with $M(t,n) = \sum_{i=1}^n D(t-i^2)$. This particular decomposition of $H(t)$ lends itself favourably to estimates for large $t$. 
Towards this end let us begin by noting $M(t,n) = O(\sqrt{t})$ ({\it see} \cite{Fricker} pp.~97). Furthermore a direct calculation gives 
\begin{align}
F(t,j) = \sum_{k=0}^m {m\choose k} (k+2)^{-1} j^{2(m-k)}(t-j^2)^{k+1}. \label{F}
\end{align} 
Thus $F(t,[ \sqrt{t} ] )M(t,[ \sqrt{t} ])  = O(t^{m+3/2})$. Additionally a straightforward application of the binomial expansions in \eqref{F} leads to,
\begin{align}
F(t,j)-F(t,j+1) = \sum_{k=1}^m {m\choose k} (k+2)^{-1} \sum_{i=0}^{k+1} \sum_{l=0}^{b_{i,k}} c_{i,k,l} j^{l} t^{k-i+1}, \label{Fdiff}
\end{align}  
where $b_{i,k} = 2i + 2(m-k)-1$ and 
\begin{equation*}
c_{i,k,l} = (-1)^{i+1} {k+1 \choose i} {2i+2(m-k) \choose k}.
\end{equation*} 
With \eqref{Fdiff} at hand and again $M(t,n) = O(\sqrt{t})$ we obtain $H(t) = O(t^{m+3/2})$. Then in a similar fashion to the above 
$L(t) = O(t^{m+3/2})$ which in conjunction with \eqref{E5} gives 
\begin{align}
\EE_5(R) & = \sum_{j=-R}^{R} \int_{\B_{r(R^2,j)}} f(x,j)\, dx + O(R^{2m+3}). 
\end{align}
To complete the proof for the base case $n=5$ we are left with showing that the leading term is as in (\ref{MRThm}). To achieve this we first 
apply the classical Euler-Maclaurin summation formula, namely,
\begin{align}
\sum_{j=-R}^{R} g(j) = \int_{-R}^{R} g(y) \, dy + \int_{-R}^R \dot{g}(y)\psi(y) \, dy, \qquad g(y) = \int_{\B_{r(R^2,y)}} f(x,y) \, dx. \label{leadingg}
\end{align}
Now recall that $f(x,y) = (\abs{x}^2+y^2)^m$ and so from a straightforward calculation 
\begin{align}
g(y) &= 2\pi^2 \left[  \frac{R^{2m+4}}{m+2} + \frac{y^{2m+4}}{(m+1)(m+2)} - \frac{y^2 R^{2m+2}}{m+1}  \right], \\
\dot{g}(y) &=   {4\pi^2}\left[\frac{{y^{2m+3}}- {y R^{2m+2}}}{m+1}  \right].
\end{align} 
Hence applying the Mean-value theorem to the second integral in \eqref{leadingg} gives
\begin{align}
\int_{-R}^R \dot{g}(y)\psi(y) \, dy &= \frac{4\pi^2}{m+1} \int_{-R}^{R} \left[{y^{2m+3}} - {y R^{2m+2}} \right]\psi(y) dy \nonumber \\
& =   \frac{4\pi^2R^{2m+3}}{m+1} \left[ \int_{\xi_1}^R \psi(y) dy -  \int_{\xi_2}^R \psi(y) dy \right] \nonumber \\
& =  \frac{4\pi^2R^{2m+3}}{m+1} \int_{\xi_1}^{\xi_2} \psi(y) dy = O(R^{2m+3}),
\end{align}
(here assuming without loss of generality that $-R<\xi_1\leq \xi_2 <R$). Then setting $\omega=(x,y)$, 
\begin{align*}
\int_{-R}^R g(y) \, dy = \int_{-R}^R \int_{\B^4_{r(R^2,y)}} (\abs{x}^2+y^2)^m\, dx \, dy = \int_{\B^5_R} \abs{\omega}^{2m} \, d\omega 
=\frac{{\rm Vol}(\Sp^4)}{2m+5}R^{2m+5}. 
\end{align*}
Thus summarising we have succeeded in proving the base case
\begin{align}
\EE_5(R) = \frac{{\rm Vol}(\Sp^4)}{2m+5}R^{2m+5}  + O(R^{2m+3}).
\end{align}
Hence we can prove $(\ref{MRThm})$ for $n\geq 5$ by induction. Indeed assume \eqref{MRThm} is true for $n$. Then as before
\begin{align}
\EE_{n+1}(R) &= \sum_{j=-R}^R  \sum_{w\in \Z^{n}} \left( \abs{x}^2+j^2\right)^m \chi_{\sqrt{R^2-j^2}}(w)  \nonumber\\
&=  \sum_{j=-R}^R \int_{\B^n_{r(R^2,j)}} f(w,j) dw + O\left(  \sum_{j=1}^R (R^2-j^2)^{2m+{n-2}}   \right) \nonumber\\
&= \frac{{\rm Vol}(\Sp^n)}{2m+n+1} R^{2m+n+1} + O\left(  R^{2m+n-1}   \right).
\end{align}
Therefore $(\ref{MRThm})$ holds for $n+1$ and thus is true for all $n\geq 5$. Note that the leading term is obtained in exactly the 
same way as the base case $n=5$ by applying the Euler-Maclaurin summation formula and the Mean value theorem.
\end{proof}

The proof of Theorem $\ref{sharpthm}$ is now a consequence of the argument preceding Theorem \ref{SharpM}. 
This can be done for $\mathbf{SO}(N)$ and $\mathbf{U}(N)$ since the lattice of analytic weights here is 
realised as $\Z^n$ whilst the multiplicity function is a homogeneous polynomial of even degree. The case of 
$\mathbf{Spin}(N)$ results from $\mathbf{SO}(N)$ as in the proof of Theorem \ref{spin}.

Note that for $n \ge 5$ we have $(n+1)/{2} \leq n-2$ and so \eqref{approx} combined with Theorem $\ref{SharpM}$ gives 
\begin{align}
\NN(\lambda) &= \kappa\frac{{\rm Vol}(\Sp^{n-1})}{(2m+n)\abs{W}} \left(\lambda + \abs{\rho}^2\right)^{m+\frac{n}{2}} 
+ O(\lambda^{m+\frac{n-2}{2}}) \nonumber \\
&= \kappa\frac{{\rm Vol}(\Sp^{n-1})}{(2m+n)\abs{W}} \lambda^{m+\frac{n}{2}} + O(\lambda^{m+\frac{n-2}{2}}).
\end{align}
The leading coefficient can be written in the following form:
\begin{align}
\kappa\frac{{\rm Vol}(\Sp^{n-1})}{(2m+n)\abs{W}} &= \frac{1}{{(2m+n)\abs{W}}} \int_{\Sp^{n-1}} 
P(\theta) d \mathcal{H}^{n-1}(\theta) \nonumber \\
&= \frac{1}{{\abs{W}}}\int_0^1 \int_{\Sp^{n-1}} P(\theta) d \mathcal{H}^{n-1}(\theta) r^{2m+n-1} dr \nonumber \\
&= \frac{1}{{\abs{W}}} \int_{\B_1^n} m_n(x) dx = \frac{{\rm Vol}\,(\G) \, \omega_d}{(2\pi)^d}. 
\end{align}
Hence we have completed the proof of Theorem \ref{sharpthm}. \hfill $\square$

\appendix

\section{Asymptotics of weighted integrals involving Bessel functions} \label{Appendix}
\setcounter {equation}{0}

In this appendix we present the proof of an estimate used earlier in the paper.  
This is concerned with the asymptotics of Bessel functions and their weighted integrals.

\begin{proposition} \label{Bessel-estimate-app}
Let $\alpha \ge 2$ and $\beta >-1/2$. Then there exist constants $M>0$ and $c>0$ such that for any $z \geq M$ we have 
\begin{equation}
\int_0^z t^{\alpha+\beta} J_{\beta}(t) \, dt \leq c z^{\alpha+\beta-\frac{1}{2}}. 
\end{equation}

\end{proposition}

\begin{proof}
Using the identity established in Lemma \ref{first-aux-lemma} below we can write
\begin{align*}
\int_{0}^{z} t^{\alpha+\beta} J_{\beta}(t) \, dt &= z^{\alpha+\beta}J_{\beta+1}(z) - (\alpha-1) \int_0^z t^{\alpha+\beta-1}J_{\beta+1}(t) \, dt. 
\end{align*}

Next in virtue of the asymptotic decay of Bessel function at infinity ({\it see}, e.g., Stein $\&$ Weiss\cite{Stein-Weiss}) it follows that there exists 
$M>0$ such that $J_{\beta+1}(z) \leq {cz^{-\frac{1}{2}}}$ for $z>M$. Hence we can write
\begin{align*}
\int_{0}^{z} t^{\alpha+\beta} J_{\beta}(t) \, dt \leq c_1 z^{\alpha+\beta-\frac{1}{2}} + c_2 \int_M^z t^{\alpha+\beta-\frac{3}{2}} \, dt 
\end{align*}
from which the conclusion follows at once. \end{proof}

\begin{lemma} \label{first-aux-lemma}
Let $\alpha \ge 2$ and $\beta >-1/2$ with $z \in \R$.Then the following identity holds
\begin{equation}
\int_{0}^{z} t^{\alpha+\beta} J_{\beta}(t) \, dt = z^{\alpha+\beta}J_{\beta+1}(z) - (\alpha-1) \int_0^z t^{\alpha+\beta-1}J_{\beta+1}(t) \, dt
\end{equation}
\end{lemma}
\begin{proof}

Starting from the following weighted integral identity for Bessel functions ({\it cf.}, e.g., Grafakos \cite{Grafakos} Appendix B.3)
\begin{align*}
\int_0^z J_{\beta-1}(t) t^\beta \, dt = z^{\beta} J_{\beta}(z) 
\end{align*}
we can write 
\begin{align*}
\int_0^z t^{\alpha+\beta-1}J_{\beta+1} (t) dt & = \int_0^z t^{\alpha-2} \int_0^t s^{\beta+1}J_{\beta}(s) \, ds dt 
%= \int_0^z t^{\alpha-2} \int_0^{\infty} \chi_{[0,t]}(s) s^{\beta+1}J_{\beta}(s) \, ds dt \\
%& = \int_0^{\infty} s^{\beta+1}J_{\beta}(s) \int_0^z \chi_{[s,\infty]}(t) t^{\alpha-2} \, dt ds 
 = \int_0^z s^{\beta+1}J_{\beta}(s) \int_s^z  t^{\alpha-2} \, dt ds. 
\end{align*}
Now integrating the term on the right gives
\begin{align*}
\int_0^z t^{\alpha+\beta-1}J_{\beta+1} (t) \, dt 
& = \frac{1}{\alpha-1}\int_0^z s^{\beta+1}J_{\beta}(s) [z^{\alpha-1}-s^{\alpha-1}] \, ds \\
& = \frac{1}{\alpha-1}z^{\alpha+\beta}J_{\beta+1}(z) - \frac{1}{\alpha-1}\int_0^z s^{\beta+\alpha}J_{\beta}(s) \, ds 
\end{align*}
and so the conclusion follows by simple manipulation of the above. 
\end{proof}

\noindent{ $\dag$ {\scriptsize DEPARTMENT OF MATHEMATICS, UNIVERSITY 
OF SUSSEX, FALMER, BRIGHTON BN1 9RF, ENGLAND, UK.}}    
\\[2 mm]
%\noindent{\textit{{\small E-mail address:}} \textrm{{\small 
%a.taheri@sussex.ac.uk}}}

\end{document}